\documentclass{article}                     
\usepackage{graphicx, amsmath, amssymb, amsfonts, xfrac, booktabs, xcolor, bm, bbm, amsthm}
\usepackage{hyperref, enumitem}
\setlist[enumerate]{label=(\roman*), leftmargin=*, labelindent=0pt}
\usepackage[margin = 1in]{geometry}
\usepackage{authblk}

\usepackage[numbers,sort]{natbib}
 \def\newblock{\ }%
 \bibpunct[, ]{[}{]}{,}{n}{}{,}%

\makeatletter
\renewcommand\subsection{\@startsection
{subsection}{2}{\z@}%
{-18\p@ \@plus -4\p@ \@minus -4\p@}%
{8\p@ \@plus 4\p@ \@minus 4\p@}%
{\normalfont\bfseries}}
\makeatother

\makeatletter
\let\cl@chapter\relax
\makeatother

\usepackage[capitalize]{cleveref}

\newcommand{\germansF}{{Po\-si\-tiv\-ste\-llen\-s\"atze}}
\newcommand{\germanF}{{Positivstellensatz}}
\newcommand{\schmu}{Schm\"{u}dgen}
\newcommand\R{{\mathbb R}}
\newcommand{\N}{\mathbb{N}}
\newcommand{\bx}{\mathbf{x}}

\newcommand{\bu}{\mathbf{u}}
\newcommand{\bg}{\mathbf{g}}

\newcommand{\cQ}{\mathcal{Q}}
\newcommand{\cT}{\mathcal{T}}
\newcommand{\cR}{\mathcal{R}}
\renewcommand{\P}{\mathcal{P}}
\newcommand{\coeff}{\operatorname{coef}}
\newcommand{\bigO}{\mathcal{O}}
\newcommand{\bigOg}{\mathcal{O}_{\bg}}
\newcommand{\Ll}{\operatorname{{\bf L}}}
\newcommand{\normsup}[1]{{\|#1\|}_{\operatorname{sup}}}
\newcommand{\normone}[1]{{\|#1\|}_{1,\coeff}}

\newcommand{\vio}{^{v}}
\newcommand{\bbg}{\bar{\mathbf{g}}}
\newcommand{\dist}{\operatorname{dist}}
\newcommand{\cl}{\operatorname{{{\bf c}}}}
\renewcommand{\Ll}{\operatorname{{\bf L}}}
\newcommand{\barg}{\bar{\mathbf{g}}}
\newcommand{\newell}{t}

\newtheorem{theorem}{Theorem}
\newtheorem{lemma}{Lemma}
\newtheorem{proposition}{Proposition}
\newtheorem{corollary}{Corollary}
\newtheorem{example}{Example}

\begin{document}

\title{
Degree Bounds for {\germansF} of general semialgebraic sets
}

\author[1]{Olga Heijmans-Kuryatnikova\thanks{\href{mailto:kuryatnikova@ese.eur.nl}{kuryatnikova@ese.eur.nl}}}
\author[2]{Juan C. Vera\thanks{\href{mailto:j.c.veralizcano@tilburguniversity.edu}{j.c.veralizcano@tilburguniversity.edu}}}
\author[3]{Luis F. Zuluaga\thanks{\href{mailto:luis.zuluaga@lehigh.edu}{luis.zuluaga@lehigh.edu}}}

\affil[1]{Econometric Institute, Erasmus University Rotterdam, Rotterdam, The Netherlands}
\affil[2]{Tilburg School of Economics and Management, Tilburg University, Tilburg, The Netherlands}
\affil[3]{Department of Industrial and Systems Engineering, Lehigh University, Bethlehem, PA, USA}

\maketitle

\begin{abstract}
Let \(p_{\min}\) denote the minimum of a polynomial \(p\) over a (general) compact semialgebraic set  \(S \subseteq \mathbb{R}^n\). A standard way to approximate \(p_{\min}\) is via hierarchies built from {\germansF}, which certify nonnegativity of polynomials on \(S\) using sums of squares or other classes of globally nonnegative polynomials. As the degree of the certificate grows, the values generated by these hierarchies converge asymptotically to \(p_{\min}\). A natural question is, then, to determine explicit bounds on the certificate's degree needed to obtain a prescribed \(\varepsilon\)-approximation to \(p_{\min}\), or equivalently certify the positivity of $f:=p - p_{\min} + \varepsilon$ on $S$.
We improve the current best degree bounds for Putinar's and  Schm\"udgen's SOS-{\germanF} over $S$. Also, we obtain degree bounds for Krivine--Stengle's and the recently introduced extended-Handelman's $\R_+$-{\germansF} over $S$; providing the first explicit degree bounds for linear optimization-based hierarchies over general compact semialgebraic sets. Our approach is based on a lift-and-project construction in which we add new variables to construct an algebraic representation of the distance to the set $S$ using {\L}ojasiewicz's inequality. This lets us lift the problem of certifying the positivity of $f$ on the (complex) set $S$ to the problem of certifying the positivity of a related polynomial $F$ on a higher-dimensional hypercube. By projecting out the added variables, non-negativity certificates for $F$ on the hypercube become non-negativity certificates for $f$ on $S$. Our approach offers a unified methodology to obtain degree bounds for several Positivstellensatz-based hierarchies over general compact sets, narrowing the gap between results for the hypercube (or other simple sets) and more general semialgebraic sets.
\end{abstract}

\section{Introduction}
\label{sec:intro}
Given a polynomial $p \in \R[\bx]:=\R[\bx_1,\dots,\bx_n]$, the set of $n$-variate polynomials with real coefficients, let
\begin{equation}\label{mod:POP}
p_{\min} :=  \min\{p(x):g_j(x) \ge 0, j=1,\dots,m\},
\end{equation}
where $g_j \in \R[\bx]$ for $j=1,\dots,m$. Let $\bg \in \R[\bx]^m$ denote the array \[\bg:=[g_1, \dots, g_m]\] and define
\[
d_\bg := \max\{\deg g_j: j=1,\dots,m\}.\]
In what follows, we assume that the semialgebraic set
\[
S_{\bg} := \{x\in \R^n: g_j(x)\ge 0, j=1,\dots,m\},
\]
is compact and nonempty.

Reformulate the {\em polynomial optimization} problem~\eqref{mod:POP} as:
\begin{equation}\label{mod:epi}
p_{\min} :=  \max\{\lambda: p(\bx) - \lambda  \in \P(S_\bg)\},
\end{equation}
where $\P(S_\bg) := \{ p \in \R[\bx]: p(x) \ge 0 \text{ for all } x \in S_\bg\} \subset \R[\bx]$ denotes the cone of polynomials nonnegative on $S_\bg$.

By approximating $\P(S_\bg)$ by smaller subsets of $\R[\bx]$ in~\eqref{mod:epi}, one obtains lower bounds on~$p_{\min}$. Such approximations can be derived from algebraic constructs defined in terms of globally nonnegative polynomials, such as {\em sum-of-squares} (SOS) and non-negative constants. Being SOS is a strong form of non-negativity, but it is {\em tractable}, as it can be formulated as a semidefinite programming (SDP) feasibility problem~\citep[see, e.g.,][]{laurent2009sums}.
To introduce these algebraic constructs,
let $\Sigma[\bx] \subset \R[\bx]$ denote the cone of all SOS polynomials; that is, the set of polynomials that can be written as $q_1^2 + \cdots + q_k^2$ for some $k \ge 1$, where each $q_j \in \R[\bx]$ for $j=1,\dots,k$. Also,
for any $\alpha \in \N^m$, denote $\bg^{\alpha} := \prod_{j=1}^m g_j^{\alpha_j}$. In particular, $\bx^{\alpha} = \prod_{i=1}^n x_i^{\alpha_i}$.

For any $r \ge 0$, we define: the \emph{truncated preorder} associated with $\bg$,
\[
\cT(\bg)_r= \left\{\sum_{\alpha \in \{0,1\}^m}\sigma_{\alpha}\bg^{\alpha}:\sigma_{\alpha} \in \Sigma[\bx], \deg(\sigma_\alpha \bg^\alpha) \le r, \alpha \in \{0,1\}^m \right\};
\]
the \emph{truncated quadratic module} associated with $\bg$,
\[
\cQ(\bg)_r=\left\{\sum_{j=0}^m \sigma_jg_j: \sigma_j \in \Sigma[\bx], \deg(\sigma_jg_j) \le r, j=0,\dots,m \right\}, \qquad (g_0 := 1);
\]
 and the \emph{truncated preprime} associated with $\bg$:
\[
\cR(\bg)_r= \left\{\sum_{\alpha \in \N^m}s_{\alpha}\bg^{\alpha}:s_{\alpha} \in \R_+, \deg(s_\alpha \bg^\alpha) \le r, \alpha \in \N^m \right\}.
\]
For all $r \ge 0$, each of these cones consists of polynomials constructed in ways that guarantee nonnegativity on $S_\bg$. In particular,
\[\cQ(\bg)_r \subseteq \cT(\bg)_r,\quad \cR(\bg)_r \subseteq \cT(\bg)_r, \text{ and } \cT(\bg)_r \subseteq \P(S_\bg).
\]
Upon selecting any of these three cones of polynomials to replace $\P(S_\bg)$ in~\eqref{mod:epi}, we obtain a hierarchy of tractable (conic) optimization problems that yields a nondecreasing sequence of lower bounds for $p_{\min}$ as~$r$ increases. Specifically, after letting
\[
p^{\cT(\bg)_r}_{\min} := \sup\{\lambda: p(\bx) - \lambda \in \cT(\bg)_r\},
\]
One obtains the hierarchy of  SDP problems:
\[
p^{\cT(\bg)_0}_{\min} \le p^{\cT(\bg)_1}_{\min} \le \cdots \le p_{\min}.
\]
Analogous  hierarchies of conic optimization problems are obtained by defining $p^{\cQ(\bg)_r}_{\min}$ and $p^{\cR(\bg)_r}_{\min}$, $r=0,1,\dots$, in a similar way. Indeed, the sequence $\smash{p^{\cQ(\bg)_r}_{\min}}$, $r=0,1,\dots$, corresponds to the Lasserre SOS hierarchy of SDP lower bounds for $p_{\min}$~\cite{Lass02b} , while the sequence $\smash{p^{\cR(\bg)_r}_{\min}}$, $r=0,1,\dots$ leads to a hierarchy of linear programming (LP) lower bounds for $p_{\min}$~\citep[see, e.g.,][]{lasserre2013lagrangian}.

Under certain compactness assumptions on $S_\bg$, \emph{\germansF}~\citep[see, e.g.,][]{laurent2009sums,lasserre2015introduction} guarantee the asymptotic convergence of these lower bound sequences to $p_{\min}$ by ensuring the existence of a nonnegative certificate for the polynomial $f(\bx) :=p(\bx)- p_{\min} + \varepsilon$ for all $\varepsilon >0$ (which is positive on $S_{\bg}$). In Theorem~\ref{thm:certificates} below, we state three examples of well-known \germansF, as well as a Positivstellensatz recently derived in~\citep[Prop. 3.5]{kuryatnikova2024reducing}, which are particularly relevant for this work.

To simplify the presentation, we introduce the following notation:  for any vector $a \in \R^l$ and an array of polynomials $\mathbf{q} = [q_1, \dots, q_l] \in \R[\bx]^l$, the notation $a-\mathbf{q}$ denotes the array of polynomials $[a_1-q_1, \dots, a_l - q_l]$. The same convention applies to expressions such as $\mathbf{q}-a$ or $a+\mathbf{q}$, and $a \pm \mathbf{q}$ denotes the array $[a_1-q_1, a_1+q_1,\dots, a_l - q_l, a_l + q_l]$. This notation may also be used when $a \in \R$, instead of $a\mathbbm{1}$, where $\mathbbm{1}$ is the vector of all ones. The expression $f > 0$ on $S_\bg$ indicates that $f(x) > 0$ for all $x \in S_{\bg}$.

\begin{theorem}
\label{thm:certificates} Let $f \in \R[\bx]$, and polynomials $g_1,\dots,g_m \in \R[\bx]$ be given such that $S_\bg = \{x \in \R^n: g_j(x) \ge 0, j=1,\dots,m\}$ is compact.
Assume  $f > 0$ on $S_\bg$.
\begin{enumerate}
\item \underline{Schm\"{u}dgen’s Positivstellensatz}~\cite{schmudgen1991k}:
\label{thm:schmudgen}
Then, $f \in \cT(\bg)_r$ for some~$r\ge 0$.
\item \underline{Putinar’s Positivstellensatz}~\cite{putinar1993positive}: Assume the \emph{quadratic module} $\cQ(\bg)_{\infty}$ is Archimedean; that is, there is $r_0 \in \N$ and $N \in \R$ such that $N - \sum_{i=1}^n \bx_i^2 \in \cQ(\bg)_{r_0}$. Then,
 $f \in \cQ(\bg)_r$ for some $r\ge 0$. \label{it:cert2}
\item \underline{Krivine-Stengle’s Positivstellensatz}~\cite[see, e.g.,][p. 25]{Marshall2008}: Assume $1-g_j \ge 0$ on $S_\bg$ for $j=1,\dots,m$ and
$\{1,g_1,\dots,g_m\}$ generates $\R[\bx]$.
Then, $f \in \cR(\bg,1-\bg)_{r}$  for some $r \ge 0$.
\label{it:KS}
\item \underline{Extended-Handelman’s Positivstellensatz}~\citep[Prop. 3.5]{kuryatnikova2024reducing}:  Assume $S_\bg \subseteq [-1,1]^n$.
Then,
$f \in \cR(1\pm \bx, \bg)_r$ for some $r\ge 0$.
\label{it:nosos}
\end{enumerate}
\end{theorem}

The impact of optimization hierarchies based on {\germansF} such as the ones in Theorem~\ref{thm:certificates} in the areas of control, robotics, combinatorics, finance, among others, is well-documented~\citep[see, e.g.,][]{magron2023sparse, nie2023moment,lasserre2015introduction}. A relevant theoretical question concerns the convergence rates of these hierarchies~\citep[see, e.g.,][]{baldi2024degree, laurent2023effective,baldi2023effective, baldi2025lojasiewicz, magron2015error, laurent2023overview, magron2025convergence}  or equivalently determining for which hierarchy level $r$ the objective value of the corresponding conic problem is within $\varepsilon > 0$ of $p_{\min}$. 
This question can be answered by obtaining {\em degree bounds} for the representation of the polynomial $f(\bx) := p(\bx) - p_{\min} + \varepsilon$ for any $\varepsilon > 0$ given by the {\germanF} associated to the hierarchy~\citep[see, e.g.,][Obs.~1]{laurent2023overview}. These bounds depend on the characteristics of $\bg$ and a {\em condition measure} of $f$, which is the ratio between an appropriate norm for $f$ and $\varepsilon = f_{\min}$. Usual norms are the sup-norm
\begin{equation}
\label{eq:supnorm}
f_{\max,D}:= \sup\{|f(x)|: x \in D\},
\end{equation}
for some appropriate set $D \subset \R^n$ and  norms on the vector of coefficients of the polynomial. Given a polynomial $f(\bx) = \sum_{\{\alpha \in \N^n: \|\alpha_1\| \le d\}} f_{\alpha} \bx^{\alpha}$ of degree $d$, we use the $\ell_1$ coefficient norm and the $\ell_{\infty}$ coefficient norm, respectively given by
\[
\normone{f} = \sum_{\{\alpha \in \N^n: \|\alpha_1\| \le d\}} |f_{\alpha}| \qquad \text{ and } \qquad \|f\|_{\infty,\coeff} = \max\{|f_{\alpha}|: \alpha \in \N^n: \|\alpha_1\| \le d\}.
\]
Also, we use the weighted coefficient norm
\begin{equation}
\label{eq:wnorm}
L(f) := \max_{\{\alpha \in \N^n: \|\alpha\|_1 \le d\}} |f_\alpha|\frac{\alpha!}{|\alpha|!}.
\end{equation}

Initial results in this direction focused on obtaining degree bounds for \germansF\ over simple sets. Specifically, these include the hypercube $[-1,1]^n$~\citep{baldi2024degree, laurent2023effective}, the binary cube $\{0,1\}^n$~\citep{slot2023sum}, the $(n-1)$-simplex $\Delta^{n-1} =\{ x \in \R^n_+: \sum_{i=1}^n x_i = 1\}$~\cite{kirschner2022convergence}, the $n$-simplex $\Delta^{n} = \{ x \in \R^n_+: \sum_{i=1}^n x_i \le 1\}$~\cite{slot2022sum}, the unit sphere $S^{n-1}: = \{x \in \R^n : \sum_{i=1}^n x_i^2 = 1\}$~\citep{fang2021sum, kirschner2022convergence}, and the unit ball $B^n = \{x \in \R^n : \sum_{i=1}^n x_i^2 \le 1\}$~\cite{slot2022sum}. Some of these results are listed in Table~\ref{tab:bounds}, using a format similar to~\citep[][Table~2.1]{slot2022asymptotic}.
\begin{table}[!ht]
\centering
\begin{tabular}{llll}
\toprule
\textbf{$S_\bg$ (compact)} & \textbf{Degree Bound} & \textbf{Positivstellensatz} & \textbf{reference} \\
\midrule
$B^{n}$ &  $\bigO\left ( \left (\frac{f_{\max, B^n}}{f_{\min}}\right )^{\tfrac{1}{2}} \right)$ & Putinar/Schm\"{u}dgen & \cite[][Thm. 3]{slot2022sum} \\
$[-1,1]^n$      & $\bigO\left(\left(\frac{f_{\max, [-1,1]^n}}{f_{\min}}\right)\right)$ & Putinar  & \cite[][Thm. 3]{baldi2024degree}\\
$[-1,1]^n$        & $\bigO\left ( \left (\frac{f_{\max, [-1,1]^n}}{f_{\min}}\right )^{\tfrac{1}{2}} \right)$ & Schm\"{u}dgen & \cite[][Cor. 3]{laurent2023effective} \\
$\Delta^{n}$ &  $\bigO\left ( \left (\frac{f_{\max, \Delta^n}}{f_{\min}}\right )^{\tfrac{1}{2}} \right)$ & Schm\"{u}dgen & \cite[][Thm. 4]{slot2022sum} \\
\midrule
Archimedean & $\bigO \left ( e^{\left ( \frac{L(f)}{f_{\min}} \right )^{c_\bg}} \right)$ & Putinar & \citep[][Thm. 6]{nie2007complexity} \\
Archimedean &  $\bigO \left ( \left (\frac{f_{\max, [-1,1]^n}}{f_{\min}} \right )^{2.5n\Ll_\bg} \right )$ & Putinar  & \citep[][Thm. 1.7]{baldi2023effective} \\
Archimedean	& $\bigO \left ( \left (\frac{f_{\max, \Delta^n}}{f_{\min}} \right )^{7\Ll_\bg +3} \right )$		         & Putinar & \citep[][Cor. 3.3]{baldi2025lojasiewicz}\\
Archimedean &  $\bigO \left ( \left (\frac{f_{\max, [-1,1]^n}}{f_{\min}} \right )^{2\Ll_\bg} \right )$ & Putinar  & see \Cref{thm:ourputi} \\
\midrule
General     &  $\bigO \left ( \left ( \frac{\| f \|_{\operatorname{\infty,\coeff}}}{f_{\min}} \right )^{c_\bg} \right)$
 & Schm\"{u}dgen & \citep[][Thm. 3]{ComplexitySchm}\\
General &  $\bigO \left ( \left (\frac{f_{\max, [-1,1]^n}}{f_{\min}} \right )^{\Ll_\bg} \right )$ & Schm\"{u}dgen  & see \Cref{thm:ourschmu} \\
General & $\bigO \left ( \left (\frac{f_{\max, [-1,1]^n}}{f_{\min}} \right )^{2\Ll_\bg} \right )$  & Krivine-Stengle  & see \Cref{thm:degbdKS} \\
General & $\bigO \left ( \left (\frac{f_{\max, [-1,1]^n}}{f_{\min}} \right )^{2\Ll_\bg} \right )$  & Extended-Handelman & see \Cref{thm:genHand}\\
\bottomrule
\end{tabular}

\caption{Relationship between the condition number of a polynomial $f > 0$  on $S_{\bg}$ and the degree bound of its associated {\germansF}, where
$\Ll_\bg$ denotes the {\L}ojasiewicz exponent associated with the set $S_\bg$ (see \eqref{eq:LojDist}).
\label{tab:bounds}}
\end{table}

Although progress has been made on deriving analogous degree bounds for {\germansF} over more general compact semialgebraic sets; as referenced in Table~\ref{tab:bounds}, such bounds are not as strong as the ones derived for simple sets, leaving the door open for potential enhancements.

The main technique used
 to obtain degree bounds for {\germansF} over general semialgebraic sets~\citep[see,][]{ComplexitySchm, baldi2025lojasiewicz, baldi2023effective} is the {\em algebro-geometric reduction} analysis~\citep[see, e.g.,][Sec. 3.2 for a detailed discussion]{laurent2023overview}. Let $f(\bx) > 0$ on $S_\bg$. To obtain a degree bound on Putinar's {\germanF} for $f$,
 first scale $S_\bg$ so it is contained in a simpler set, such as the hypercube $[-1,1]^n$. Next, using bounds given by the {\em {\L}ojasiewicz inequality}~\citep{loja} (see Section~\ref{sec:loja} for details), construct a polynomial $q \in \cQ(\bg)_{\newell}$ such that $f-q > 0$ on $[-1,1]^n$. Then, apply existing {\schmu}'s {\germanF} degree bounds over the hypercube, to find~$r$ such that $f-q \in \cT(1 \pm \bx)_r$. Finally, use a relationship of the form $\cT(1 \pm x)_r \subseteq \cQ(\bg)_{r+\newell'}$ which implies the desired degree bound $f = (f-q) + q \in \cQ(S_{\bg})_{\max\{r+\newell', \newell\}}$. Improved degree bounds obtained using this approach~\cite{, baldi2025lojasiewicz, baldi2023effective} can be mainly attributed to the way in which the polynomial $q$ above is constructed~\cite[][Sec. 3.2.1]{laurent2023overview}.

\subsection*{Contributions}
We derive new degree bounds for {\germansF} over general compact semialgebraic sets (see Table~\ref{tab:bounds}). To accomplish this, we propose a {\em lift-and-project} approach for polynomial optimization. This approach leverages degree bounds for {\germansF} on the hypercube but differs fundamentally from algebro-geometric reduction methodologies. The approach is motivated by the results in~\cite{kuryatnikova2024reducing}. It adopts principles from classical mixed-integer programming (MIP) in which the MIP is lifted into a higher-dimensional space by adding new variables. This variables capture interactions between original variables, and by imposing relaxed consistency conditions they allow to represent convex hulls more accurately. Projecting back to the original space stronger LP relaxations of the MIP are obtained. ~\citep[see, e.g.,][]{BalaCC93, sherali2013reformulation}.

Section~\ref{sec:liftintro} details and illustrates the lift-and-project approach. Here, we summarize:
We begin by scaling $S_\bg$ so it is contained in the hypercube $[-1,1]^n$. Next, given $f \in \R[x]$ positive on $S_{\bg}$, we add $m$ new variables $u_1, \ldots, u_m$, one for each constraint $g_j(x) \ge 0$, $j=1,\dots,m$. We construct $F \in \R[\bx,\bu]$ as a \emph{lifted} version of $f$. This lifted polynomial is positive on $T := [-1,1]^n\times[0,1]^m$. The key aspect of this construction is that eliminating the new variables by the substitution $\bu = \bg(\bx)$ projects $T$ onto $S_g$ and $F(\bx,\bu)$ onto $f(\bx)$. More specifically, $S_\bg =\{x \in \R^n: (x,\bg(x)) \in T\}$, and $f(\bx) = F(\bx,\bg(\bx))$. By applying an effective {\germanF} over $T$, we obtain a certificate of nonnegativity for $F$ over $T$ with bounded degree. After substituting $\bu = \bg(\bx)$, we project this certificate to a certificate of the same type for the nonnegativity of~$f$ over $S_\bg$. For example, if $F \in \cQ(1\pm \bx,\bu,1-\bu)_r$, the method yields $f\in \cQ(\bg)_{ar+b}$, where $a$ and $b$ depend only on $\bg$.

We apply the lift-and-project approach to obtain new degree bounds (see Table~\ref{tab:bounds}). Let $f >0$ on $S_\bg$ and $\Ll_\bg$ denote the {\L}ojasiewicz exponent of $\bg$ (see \eqref{eq:LojDist}).
\begin{enumerate}
\item We improve the current best Putinar's Positivstellensatz degree bound from $\bigO((f_{\max,\Delta^n}/f_{\min})^{7\Ll_\bg+3})$~\citep[][Cor. 3.3]{baldi2025lojasiewicz}  to $\bigO((f_{\max,[-1,1]^n}/f_{\min})^{2\Ll_\bg})$ (\Cref{thm:ourputi}).
\item We obtain a $\bigO ( (f_{\max, [-1,1]^n}/f_{\min} )^{\Ll_\bg} )$ degree bound on Schm\"{u}dgen's Positivstellensatz (\Cref{thm:ourschmu}). Notably, the degree bounds for Putinar's and {\schmu}'s {\germansF} over general semialgebraic sets match the same quadratic improvement as for the hypercube $[-1,1]^n$ (see Table~\ref{tab:bounds}).
\item We obtain $\bigO ( (f_{\max, [-1,1]^n}/f_{\min}  )^{2\Ll_\bg}  )$ degree bounds on Krivine-Stengle's Positivstellensatz (\Cref{thm:degbdKS}) and the extended-Handelman's Positivstellensatz (\Cref{thm:genHand}). These results provide the first explicit degree bounds for linear optimization-based hierarchies over general compact semialgebraic sets that use nonnegative constants to certify nonnegativity. Interestingly, the degree bounds are similar to those for Putinar's and Schm\"{u}dgen's {\germansF} which use SOS to certify nonnegativity.
\end{enumerate}

Above, it is important to note that  $\Ll_\bg =1$ holds under common regularity assumptions~\citep[][Thm.~2.11]{baldi2023effective}.

\subsection*{Outline} The article is organized as follows. Section~\ref{sec:prelim} begins by presenting basic algebraic results for nonnegative polynomials. We then present results on the {\L}ojasiewicz inequality~\cite{loja}. The section finishes with an illustrative description of the lift-and-project procedure. This procedure is formalized in Section~\ref{sec:main} (see Theorem~\ref{thm:enhLift}), where we apply this generic result to obtain improved degree bounds for Putinar (Section~\ref{sec:Putinar}), Schm\"{u}dgen's (Section~\ref{sec:Schmu}), Krivine-Stengle's, and extended-Handelman's (Section~\ref{sec:both}) {\germansF}. Section~\ref{sec:final} concludes with final remarks and future directions.

\section{Preliminaries}
\label{sec:prelim}
In this section, we first (Section~\ref{sec:loja}) introduce tools associated to the {\L}ojasiewicz inequality~\cite{loja} --- commonly employed in related articles~\citep[see, e.g.][]{baldi2023effective, baldi2025lojasiewicz} --- and use it to develop a semialgebraic description of the distance to the set $S_{\bg}$  that is crucial to derive the degree bounds presented here (see~\eqref{eq:Hdef}). Then, we draw on the approach by \citet{kuryatnikova2024reducing} to introduce the lift-and-project procedure for polynomial optimization.
This procedure is fundamental to obtain our main result (\Cref{thm:enhLift}) which is used to derive degree bounds for different {\germansF} for general compact semialgebraic sets (Sections~\ref{sec:Putinar}--~\ref{sec:both}).

From now on, for any polynomial $f$, it is convenient to use $\normsup{f}$ to denote the sup-norm $f_{\max,[-1,1]}$ and
\begin{equation}
\label{eq:kappadef}
\kappa_{S_\bg}(f) = \tfrac{\normsup{f}}{f_{\min}},
\end{equation}
to denote the condition number of a polynomial $f > 0$  on $S_{\bg}$. Notice that for any compact set $S_\bg \subseteq [-1, 1]^n$,  we have
\begin{equation}
\label{eq:kapparange}
1 \le \kappa_{S_\bg}(p) < \infty.
\end{equation}

\subsection{Algebra of nonnegativity certificates}
We present some basic algebraic results related to nonnegativity certificates.

\begin{lemma}\label{lem:SubsQTR} Let $r_0 \ge 0$. Let $\bg \in \R[x]^m$ and $\mathbf{h} \in \R[x]^{m'}$.
\begin{enumerate}
\item If $\mathbf{h} \subset \cQ(\bg)_{r_0}$, then
\( \cQ(\bg,\mathbf{h})_r \subseteq \cQ(\bg)_{r_0+r}.
\) for all $r >0$. \label{it:lem1}
\item If $\mathbf{h}\subset \cT(\bg)_{r_0}$, then
\( \cT(\bg,\mathbf{h})_r \subseteq \cT(\bg)_{r_0 r}
\)
 for all $r >0$. \label{it:lem2}
 \item If $\mathbf{h}\subset \cR(\bg)_{r_0}$, then
\( \cR(\bg,\mathbf{h})_r \subseteq \cR(\bg)_{r_0 r}
\)
 for all $r >0$. \label{it:lem3}
\end{enumerate}
\end{lemma}

The proof of the following lemma readily follows from the proofs of similar related results in~\citep[][Thm. 3.3]{de2026link} and ~\citep[][Lem. 2.7]{gribling2026revisiting}.

\begin{lemma}
\label{lem:norm1f-f}
Let $f \in \R[\bx]$ with $\deg f = d$. Then
\begin{enumerate}
\item $\normone{f} - f  \in \cR(1\pm \bx)_{d}$ \label{normfit:one}
\item $\normone{f}  - f  \in \cT(1\pm \bx)_{d}$  \label{normfit:two}
\item $\normone{f}  - f  \in \cQ(1\pm \bx)_{2d+1}$  \label{normfit:three}
\end{enumerate}
\end{lemma}
\begin{proof}
For any $f \in \R[\bx]$ with $\deg f = d$,
\begin{equation}
\label{eq:expand}
\normone{f}  - f = \sum_{\{\alpha \in \N^n:\|\alpha\|_1 \le d\}} (|f_{\alpha}| - f_{\alpha} \bx^{\alpha}) = \sum_{\{\alpha \in \N^n:\|\alpha\|_1 \le d\}}  |f_{\alpha}|(1 - \operatorname{sign}(f_{\alpha})\bx^{\alpha}),
\end{equation}
thus, it is enough to show the results for $f = 1 \pm \bx^\alpha$.

For any $i \in \{1,\dots,n\}$ and $\alpha \in \N^n$ we have 
\[1 + \bx_i \bx^\alpha = \tfrac{1}{2}[(1 + \bx_i)(1+\bx^\alpha)+ (1 - \bx_i)(1-\bx^\alpha)] 
\text{ and } 1 - \bx_i \bx^\alpha = \tfrac{1}{2}[(1 - \bx_i)(1+\bx^\alpha)+ (1 + \bx_i)(1-\bx^\alpha)]  .\]
 By an inductive argument on $\|\alpha\|_1$ we obtain $1 \pm \bx^\alpha \subseteq \cR(1 \pm x)_{\|\alpha\|_1}$, proving~\ref{normfit:one}. By $\cR(1\pm \bx)_{d} \subset \cT(1\pm \bx)_{d}$ we obtain~\ref{normfit:two}.
To prove~\ref{normfit:three} we write
\[ 1 + \bx^\alpha = \tfrac{1}{2}[ (1 + \bx^\alpha)^2 + ( 1 - \bx^{2\alpha})]
 \text{ and }  1 - \bx^\alpha = \tfrac 12 [ (1 - \bx^\alpha)^2 + ( 1 - \bx^{2\alpha})].\]
Thus, it is enough to show $ 1 - \bx^{2\alpha} \in \cQ(1\pm \bx)_{2\|\alpha\|_1+1}$. We do this by induction. For the base case, notice that $1 - \bx_i^2 = \frac 12 \left( (1-\bx_i) (1 + \bx_i)^2 + (1+\bx_i)( 1 - \bx_i^{2}) \right)$ and for the inductive step, we use
\[
1 - \bx_i^2\bx^{2\alpha}  = (1-\bx_i^2) + \bx_i^2(1-\bx^{2\alpha}).
\]
\end{proof}

The next lemma is key to apply effective {\germansF} for $[-1,1]^n$ or $[0,1]^n$ to obtain effective {\germansF} on other hypercubes.

\begin{lemma}\label{lem:trans}
Let $S\subseteq \R^n$, $\bg \in \R[\bx]^m$, and $Z: \R^n \to \R^n$ be a polynomial transformation.
\begin{enumerate}
\item If $f(\bx)> 0$ on $Z(S)$ then $f(Z(\bx)) > 0$ on $S$. \label{it:trans1}
\item If $f\in \cQ(\bg)_r$ then $f(Z(\bx)) \in \cQ(\bg(Z(\bx))_{r\deg Z}$. \label{it:trans2}
\end{enumerate}
\end{lemma}

\subsection{{\L}ojasiewicz bounds for polynomial optimization}
\label{sec:loja}
The {\L}ojasiewicz inequality~\cite{loja} is a commonly used tool in the derivation of effective Positivstellensatz  over generic compact semialgebraic sets~\citep[see, e.g.][]{baldi2023effective, baldi2025lojasiewicz}.

Given $g_1,\dots,g_m \in \R[x]$ such that $S_\bg$ is compact, define for any $x \in \R^n$ the \emph{violation} components
\[
g_j(x)\vio := \max \left \{0, -g_j(x) \right \},\qquad j=1,\dots,m,
\]
and the \emph{maximum violation} function
\begin{equation}
\label{eq:Gdef}
G(x) := \|[g_1(x)\vio, \dots, g_m(x)\vio]\|_\infty.
\end{equation}
which is continuous and semialgebraic.
Notice that $ G^{-1}(0) = S_\bg$. In particular, $G(x) = 0$ implies $\dist(x,S_\bg)=0$. The function~$G$ is referred as the algebraic distance to $S_\bg$.
We denote by $\Ll_\bg$ and $\cl_\bg$ the {\L}ojasiewicz exponent and constant corresponding to the pair $\dist(\cdot,S_\bg)$ and $G$ on $[-1,1]^n$, given by the {\em {\L}ojasiewicz inequality}~\citep{loja}:
\begin{equation}\label{eq:LojDist}
\dist(x,S_\bg)^{\Ll_\bg} \le \cl_\bg G(x).
\end{equation}

As the next lemma shows, the {\L}ojasiewicz exponent $\Ll_\bg$, which is the main parameter used to compare different degree bounds in Table~\ref{tab:bounds},
is invariant under positive scalings of $\bg$~\citep[see also, e.g.,][]{baldi2025lojasiewicz, baldi2023effective}.

\begin{lemma}\label{lem:scaled-loj}
Let \(c \in \R^m_{++}\) and \(\bg=[g_1,\dots,g_m] \in \R[\bx]^m\) be given. Let \(\tilde{\bg}:=[\tilde g_1,\dots,\tilde g_m]\) be defined by setting
$\tilde g_j=c_j  g_j$ for $j=1,\dots,m$. Then, 
\[
\Ll_{\tilde \bg}=\Ll_{\bg} \text{ and }
\frac{1}{\max_{j=1,\dots,m} c_j} \cl_{\bg} \le \cl_{\tilde \bg} \le \frac{1}{\min_{j=1,\dots,m} c_j}\cl_{\bg}.
\]
\end{lemma}

\begin{proof}
Let $c_{\min} = \min_{j=1,\dots,m} c_j$. Define, 
\[
\tilde G(\bx) := \|\max \left \{0, - \tilde g_1(x) \right \}, \dots, \max \left \{0, - \tilde g_m(x) \right \}\|_\infty.
\ge 
c_{\min} G(\bx) 
\]
Using \(S_{\tilde \bg}=S_{\bg}\), 
\begin{equation}\label{eq:LgScaling}
\dist(\bx,S_{\tilde \bg})^{\Ll_{\bg}}
=
\dist(\bx,S_{\bg})^{\Ll_{\bg}}
\le
\cl_{\bg} G(\bx)
\le
\tfrac{1}{c_{\min}} \cl_{\bg} \tilde G(\bx),
\end{equation}
which shows that $\Ll_{\tilde \bg} \le \Ll_\bg$. 
Interchanging the roles of $\bg$ and $\tilde \bg$, using $g_j=\frac 1{c_j}  \tilde g_j$ for $j=1,\dots,m$,
we also obtain $\Ll_{\bg} \le \Ll_{\tilde \bg}$, which shows $\Ll_{\bg} = \Ll_{\tilde \bg}$.

Also, equation~\eqref{eq:LgScaling} gives the desired expression for the {\L}ojasiewicz constant.
\end{proof}

Now, for any $x \in \R^n$, consider the squared Euclidean norm of the violation vector
\begin{equation}
\label{eq:Hdef}
H(x): = \|[g_1(x)\vio, \dots, g_m(x)\vio]\|^2_2 = \sum_{j=1}^m (g_j(x)\vio)^2.
\end{equation}
From the fact that for any $a \in \R^m$, $\|a\|_\infty^2 \le \|a\|_2^2  \le m \|a\|_\infty^2$, we obtain
\begin{equation}\label{eq:norms}
G(x)^2 \ \le\ H(x) \ \le\ m  G(x)^2.
\end{equation}

Using these facts, we obtain the following error bound on the minimum of a polynomial.
\begin{proposition} \label{prop:bound}
Let $f \in \R[\bx]$ with $\deg f = d$, and $g_1,\dots,g_m \in \R[\bx]$ such that $S_\bg \subseteq [-1,1]^n$ and nonempty. If $f_{\min} > 0$, then for all $x \in [-1,1]^n$,
\[
\max\{0,f_{\min}-f(x)\}
\le \tfrac{1}{2}\cl_\bg^2(4d^2\kappa_{S_{\bg}}(f))^{2\Ll_\bg}f_{\min} H(x) + \tfrac 12 f_{\min} .
\]\end{proposition}

Before proving Proposition~\ref{prop:bound}, we present some auxiliary results.
\begin{lemma}
\label{lem:gx} Let $f \in \R[\bx]$. Let $d = \deg f$. Let $g_1,\dots,g_m \in \R[\bx]$ be such that $S_\bg \subseteq [-1,1]^n$ and nonempty. Then, for all $x \in [-1,1]^n$,
\[
\max\{ 0, f_{\min} - f(x)\} \le d(2d-1)\dist(x,S_\bg)\normsup{f}.
\]
\end{lemma}
\begin{proof}
Let $x\in [-1,1]^n$ be given. If $f(x) \ge f_{\min}$ the statement is trivial. Thus, we assume $f(x) < f_{\min}$. Let $y \in S_\bg$ be such that $\dist(x,S_\bg) = \|x-y\|$.
From the mean value theorem, $f(x) - f(y) = (x-y)^T \nabla f(\alpha x + (1-\alpha)y)$ for some $\alpha \in [0,1]$.
From Cauchy-Schwartz $|f(x) - f(y)| \le \|x-y\|_2 \|\nabla f(\alpha x + (1-\alpha)y)\|_2$.
Using \cite[Thm.~3]{kroo} we have $\|\nabla f(\alpha x + (1-\alpha)y)\|_2 \le \normsup{\|\nabla f\|_2} \le d(2d-1)\normsup{f}$.
Therefore, $0 < f_{\min} - f(x) \le f(y) - f(x)  \le d(2d-1)\dist(x,S_\bg)\normsup{f}$. 
\end{proof}

\begin{corollary}
\label{cor:Lgbound} Let $\bg \in \R[\bx]^m$ be such that $S_\bg \subseteq [-1,1]^n$ and nonempty. Then
the {\L}ojasiewicz exponent $\Ll_\bg$ in~\eqref{eq:LojDist} satisfies $\Ll_\bg \ge 1$.

\begin{proof}
Take $j \in \{1,\dots,m\}$ and note that $\min_{x \in S_\bg} g_j(x) \ge 0$.  Now take $x \in [-1,1]^n$.
By \Cref{lem:gx} applied to $g_j$ with degree $\deg g_j = d_j$, we obtain
\[
d_j(2d_j-1)\normsup{g_j}\dist(x,S_\bg)  \ge \max \left \{ 0, \min_{x \in S_\bg} g_j(x) - g_j(x) \right \} \ge \max\{ 0,  -  g_j(x)\} = g_j(x)\vio.
\]
Thus, $G(x) \le 2d_\bg^2\max_{j=1, \dots,m}\normsup{g_j}\dist(x,S_\bg)$. From~\eqref{eq:LojDist},
it then follows that~$\Ll_\bg \ge 1$. 
\end{proof}
\end{corollary}

\begin{lemma}\label{lem:bound x ala alfa}
Let $0<\alpha < 1$ and $b>0$. Then,
\[
x^\alpha \le {b}^{-\frac{1-\alpha}{\alpha}}x+b \qquad \text{for all } x > 0
\]
\end{lemma}

\begin{proof}
Let $x>0$. By AM-GM inequality, we have
\(
 b^{\frac {1-\alpha}{\alpha}}x^{\alpha} \le \alpha x + (1-\alpha)b^{\frac 1{\alpha}} ,
\)
which implies  $x^\alpha \le \alpha{b}^{-\frac{1-\alpha}{\alpha}}x+ (1-\alpha)b \le {b}^{-\frac{1-\alpha}{\alpha}}x+b$. 
\end{proof}

We now present the proof of \Cref{prop:bound}.

\begin{proof}[Proof of \Cref{prop:bound}]
Take $x  \in [-1,1]^n$.
We have
\begin{align}
\notag \max\{0, f_{\min} - f(x)\}
& \le  d(2d-1) \normsup{f}\dist(x,S_\bg)
& \text{(by~\Cref{lem:gx})}\\
\notag & \le 2d^2 \normsup{f}(\cl_\bg G(x))^{1/\Ll_\bg} & (\text{by~\eqref{eq:LojDist}})\\
& \le
2d^2\normsup{f}\cl_\bg^{1/\Ll_\bg}H(x)^{\tfrac{1}{2\Ll_\bg}} & \text{(by~\eqref{eq:norms})}. \label{eq:pmin-p}
\end{align}
Let $c =2d^2\normsup{f}\cl_\bg^{1/\Ll_\bg}$. By \Cref{cor:Lgbound}, $\Ll_\bg \ge 1$. Then from
\Cref{lem:bound x ala alfa} we have
\begin{equation}\label{eq:Hala}
H(x)^{\tfrac{1}{2\Ll_\bg}} \le  \left(\frac{2c}{f_{\min}}\right)^{2\Ll_\bg-1} H(x) + \frac{f_{\min}}{2c}.
\end{equation}
Combining~\eqref{eq:pmin-p} and~\eqref{eq:Hala} we obtain,
\begin{align*}
\max\{0,f_{\min}-f(x)\}
\le    c H(x)^{\tfrac{1}{2\Ll_\bg}}
&\le \tfrac{1}{2}\left(\frac{2c}{f_{\min}}\right)^{2\Ll_\bg} f_{\min}\; H(x) + \tfrac 12  f_{\min}\\
&= \tfrac{1}{2}\cl_\bg^2(4d^2\kappa_{S_{\bg}}(f))^{2\Ll_\bg} f_{\min}H(x) + \tfrac 12f_{\min}.
\end{align*}
\end{proof}

\subsection{Lift-and-project for polynomial optimization}
\label{sec:liftintro}
Lift-and-project methods are central in optimization. These methods lift a complex set $S \subseteq \mathbb{R}^n$ to a higher-dimensional, typically convex or polyhedral, set $T \subseteq \mathbb{R}^{n+k}$. After lifting, further operations—such as adding constraints—are performed on $T$. Projecting $T$ back to $\mathbb{R}^n$ yields a tighter or more structured relaxation of $S$. This viewpoint underlies extended formulations in polyhedral combinatorics~\cite{Schrijver1986, GLS1988}, classical lift-and-project procedures in integer optimization~\cite{balas1993lift}, and semidefinite relaxations of nonconvex quadratic problems~\cite{BenTalNemirovski2001}. It offers a unified geometric framework for strengthening relaxations and revealing hidden convexity.

Recently, in~\citep{kuryatnikova2024reducing,roebers2021sparse}, a lift-and-project methodology for polynomial optimization is used to obtain new types of non-SOS Positivstellensatz over both compact and noncompact sets. We next illustrate this methodology, which is the core idea used to obtain our results.

Given $f \in \R[\bx]$ and $\bg \in \R[\bx]^m$, a {\em lifting} of $(f,S_\bg)$ is a pair $(F,T)$ with $\{(x,\bg(x)):x\in S_\bg\} \subseteq T \subseteq \R^n \times \R^m_+$ and $F \in \R[\bx,\bu]$ such that the substitution $\bu = \bg(\bx)$ projects $T$ onto $S_{\bg}$, and projects $F(\bx,\bu)$ onto $f(\bx)$. Specifically, $S_{\bg} =\{x \in \R^n: (x,\bg(x)) \in T\}$ and $f(\bx) = F(\bx,\bg(\bx))$. If $F \in \R[\bx,\bu]$ is positive (resp. nonnegative) on $T$, we call $(F,T)$ a  {\em positive} (resp. {\em nonnegative}) lifting. By definition, a lifting $(F,T)$ of $(f,S_\bg)$ implies that for all $x \in S_\bg$, the point $(x,\bg(x)) \in T$ and thus $p(x) = F(x,\bg(x)) >0$ (resp. $p(x) = F(x,\bg(x)) \ge 0$). That is, if a positive (resp. nonnegative) lifting $(F,T)$ of $(f,S_\bg)$ exists, then $f$ is positive (resp. nonnegative) on $S_\bg$. This is illustrated in \Cref{ex:lift2} next.

\begin{example}\label{ex:lift2}
Define the  polynomials
$g_1(\bx)=\bx_1$, $g_2(\bx)=\bx_2$ and $g_3(\bx) = 1- (1+\bx_1)(1+\bx_2)/2$, and let $\bg = [g_1, g_2,g_3]$.
Then,
\(
S_\bg \subseteq [-1,1]^2.
\)
To certify the nonnegativity of $f(\bx)=\frac 34-(\frac{1}{2}-\bx_1)^2-(\frac{1}{2}-\bx_2)^2$ on $S_\bg$, introduce the polynomial
\begin{equation}
\label{eq:Fex}
F(\bx,\bu) =\tfrac 14 + 2\bigl(\bu_1\bu_2 + \bu_1\bu_3+\bu_2\bu_3\bigr)  + \bx_2^2 \bu_1 + \bx_1^2 \bu_2
\end{equation}
and define the set
\(
T= [-1,1]^2 \times [0,1]^3 \subset \R^2\times \R^3_+.
\)
Clearly, $F(x,u)\ge 1/4 > 0$ for any $(x,u) \in T$. Applying the simple polynomial substitution $\bu_j = g_j(\bx)$, $j=1,2,3$,  an expanding the expression we obtain
\begin{equation}\label{eq:exNA}
F(\bx,\bg)
= \tfrac 14 + 2(g_1(\bx)g_2(\bx) + g_1(\bx)g_3(\bx)+g_2(\bx)g_3(\bx)) + \bx_1^2 g_2(\bx) + \bx_2^2 g_1(\bx)
= f(\bx).
\end{equation}
Therefore, $(F,T)$ is a positivity lifting of $(f,S_{\bg})$, which shows $f>0$ on $S_{\bg}$.
\end{example}

As shown in~\cite[][Thm. 3.2]{kuryatnikova2024reducing} every polynomial $f$  positive on $S_\bg$ has a positive lifting. More precisely, Given $T$ such that $\{(x,g(x)):x\in S_\bg\} \subseteq T \subset \R^n \times \R^m_+$, there is $F$ of degree at most $\max(\deg f,2d_\bg)$, such that $(T,F)$ is a positive lifting of $(f,S_\bg)$.

The key property of these liftings is that not only the positivity of $F$ on $T$ implies the positivity of $f$ on $S_{\bg}$. Also, the substitution $\bu = \bg(\bx)$ projects any nonnegativity certificate for $F$ on $T$ into a nonnegativity certificate for $p$ on $S_{\bg}$.
\begin{example}[\Cref{ex:lift2} revisited] Consider $f$, $\bg$, $F$, and $T$ defined in \Cref{ex:lift2}. Now, we illustrate how to use certificates for $F>0$ on $T = [-1,1]^2\times[0,1]^3$ to produce certificates for $f>0$ on $S_{\bg}$.
First, by~\eqref{eq:Fex}, we have $F \in \cT(1 \pm \bx,\bu)_3$. The substitution $\bu = \bg(\bx)$ shows that $f = F(\bx,\bg(\bx)) \in \cT(1 \pm \bx,\bg)_6$ as $d_\bg \le 2$. Notice that in~\eqref{eq:exNA} we actually obtain $f \in \cT(\bg)_3$. Moreover, as $F$ positive on $[-1,1]^2 \times [0,1]^3$ we know from Putinar's {\germanF} (see \Cref{thm:certificates}\ref{it:cert2}) that $F \in \cQ(1 \pm \bx,\bu,1-\bu)_r$ for some $r \ge 0$, which implies $f(\bx) = F(\bx,\bg) \in \cQ(1\pm \bx,\bg,1-\bg)_{2r}$.
For instance
\[
\begin{aligned}
F(\bx,\bu) &=(\bu_1+\bu_2+\bu_3-\tfrac12)^2
+\bigl((1-\bu_1)^2+\bx_2^2\bigr)\bu_1
+\bu_1^2(1-\bu_1)\\
&\quad+\bigl((1-\bu_2)^2+\bx_1^2\bigr)\bu_2
+\bu_2^2(1-\bu_2)+(1-\bu_3)^2\bu_3
+\bu_3^2(1-\bu_3),
\end{aligned}
\]
shows we can take $r = 3$.
\end{example}

Next, we show how the lift-and-project approach illustrated above can be used to obtain
effective {\germansF} over $S_{\bg}$.

\section{Effective {\germansF}}
\label{sec:main}

Our strategy to obtain nonnegativity certificates unfolds in three steps: scaling, lifting, and substitution.
Given a polynomial \(f>0\) on \(S_{\bg}\), we first scale the variables so that \(1\pm \bx \in \mathcal{R}_{r_0}(\mathbf g)\)
and simultaneously scale the defining
constraints so that \(1-\bg \in \mathcal{R}_{r_0}(\bg)\).
In particular, this normalization ensures that \(S_\bg \subseteq [-1,1]^n\) and
\(\|g_j\|_{\sup}\le 1\), $j=1,\dots,m$.
 We then construct a positive lifting \((F,T)\) of
\((f,S_\bg)\) with lifting set \(T=[-1,1]^n\times[0,1]^m\). Applying an effective
Positivstellensatz on the hypercube yields a bounded-degree certificate of nonnegativity
for \(F(\bx,\bu)\) in terms of the generators \(1\pm \bx\),
\(\bu\), and \(1-\bu\), with degree bounded by some \(R\).
Substituting \(\bu\) by \(\bg(\bx)\) and using the identity
\(f(\bx)=F(\bx,\bg(\bx))\) then provides a corresponding representation
of \(f\) in terms of \(1\pm \bx\), \(\bg\), and \(1-\bg\).
Finally, thanks to the initial scaling, the quantities \(1\pm \bx\) and
\(1-\bg\) can themselves be expressed in terms of \(\bg\), yielding
the desired effective representation of \(f\).

After all the substitutions, the final degree \(r\) of the nonnegativity certificate for \(f\) takes the form \(r = aR + b\), where \(a\) and \(b\) are constants depending only on \(\bg\) and \(R\) depends on \(n+m\), \(\deg(F)\), and the \emph{condition measure} $\frac{\|F\|}{F_{\min}}$ of the lifting \(F\). In Theorem~\ref{thm:enhLift}, we refine~\cite[][Thm.~3.2]{kuryatnikova2024reducing} by additionally deriving a bound on the condition number of $F$, in terms of the condition number of $f$.

\begin{theorem}
\label{thm:enhLift}
Let $f \in \R[\bx]$ with $\deg f = d$, and $g_1,\dots,g_m \in \R[\bx]$ with $\normsup{g_j} \le 1$, $j=1,\dots,m$, such that $S_\bg \subseteq [-1,1]^n$ and nonempty . Choose a compact set $T \subseteq [-1,1]^n \times [0,1]^m$ such that
\(
\{(x,\bg(x)) : x \in S_\bg\} \subseteq T
\). Then  \(
f > 0
\) on $S_\bg$ if and only if there is $F \in \R[\bx, \bu]$ with degree $\deg F = \max\{d, 2d_\bg\}$ such that
\begin{enumerate}[label=(\roman*)]
\item $f(\bx) = F(\bx,\bg(\bx))$, \label{it:T2}
\item $\min_{(x,u)\in T}F(x,u) \ge \tfrac 12 f_{\min}$,\label{it:T1}
\item $\max_{(x,u)\in T} F(x,u) \le (1+ 2m\cl_\bg^2)(4d^2 \kappa_{S_\bg}(f))^{2\Ll_\bg}f_{\min}$. \label{it:T3}
\end{enumerate}
\end{theorem}
\begin{proof}
In one direction, if such $F$ exists, then $f(x) = F(x,\bg(x))> 0$ for all $x \in S_\bg$ since $\{(x,\bg(x)) : x \in S_\bg\}  \subseteq T$ and $F >0$ on $T$. Conversely, assume $f > 0$ on $S_\bg$, then $f_{\min} := \min_{x \in S_\bg} f(x) > 0$. Denote $\|\bu-\bg(\bx)\|_2^2 :=  \sum_{j=1}^m (u_j - g_j(\bx))^2$ and set
\[
F(\bx,\bu) := f(\bx) + \hat\cl_{\bg,f}\|\bu-\bg(\bx)\|_2^2 \text{ with }\hat\cl_{\bg,f} := \tfrac{1}{2}\cl_\bg^2(4d^2\kappa_{S_\bg}(f))^{2\Ll_\bg}f_{\min},
\]
Note that $\deg F  = \max\{d, 2d_\bg\}$. Further,  $F(\bx,\bg(\bx)) = f(\bx)$, so~\ref{it:T2} holds.
To prove~\ref{it:T1}, we first notice that since $\normsup{g_j} \le 1$, $j=1,\dots,m$, it follows that for any $x \in [-1,1]^n$,
\begin{equation}
\label{eq:Hopt}
H(x)  = \|\bg(x)\vio\|^2_2 = \min_{u \in [0,1]^m}\|u-\bg(x)\|_2^2.
\end{equation}
Thus, for any $(x,u) \in T$ we have
\begin{align*}
F(x,u)
& \ge f(x) + \hat\cl_{g,f} \min_{u \in [0,1]^m}\|u-\bg(x)\|_2^2 \\
& = f(x) + \hat\cl_{\bg,f} H(x)
& \text{(by~\eqref{eq:Hopt})}\\
& \ge f(x) + \max\{0, f_{\min}-f(x)\} - \tfrac 12 f_{\min} &\text{(by~\Cref{prop:bound})}\\
& \ge \tfrac 12 f_{\min}.
\end{align*}
Furthermore, we have,
\begin{align*}
\max_{(x,u)\in T} F(x,u)
& \le \max_{x \in [-1,1]} f(x)  + \hat\cl_{\bg,f} \max_{(x,u)\in T} \|u-\bg(x)\|_2^2\\
& \le \normsup{f} + \hat\cl_{\bg,f}\sum_{j=1}^m(1+\normsup{g_j})^2\\
& \le \normsup{f} + 4m\hat\cl_{\bg,f}\\
& \le (1+ 2m\cl_\bg^2)(4d^2\kappa_{S_\bg}(f))^{2\Ll_\bg}f_{\min}, &  \text{(from~\eqref{eq:kapparange} and \Cref{cor:Lgbound})}
\end{align*}
proving~\ref{it:T3}. 
\end{proof}

In the sections that follow, we leverage Theorem~\ref{thm:enhLift} to
improve the current best Putinar's Positivstellensatz degree bound
and to obtain new effective versions of Schm\"{u}dgen's, Krivine-Stengle's, and extended-Handelman's {\germansF}.

In general, the degree bounds for a polynomial $f > 0$ on $S_{\bg}$ will depend both on the characteristics of $f$ and $\bg$. As in related literature, the focus in the results that follow is to characterize the dependance of the degree bounds on the characteristics of $f$, such as the number of variables $n$, the $\deg f = d$, and the condition number of $f$, $\kappa_{S_{\bg}}(f)$~\eqref{eq:kappadef}. For that purpose we will use the $\bigO$-notation
$\bigOg(\cdot)$
to indicate that an expression is bounded up to multiplicative constants that depend only on $\bg$.

\subsection{Effective Putinar Positivstellensatz}
\label{sec:Putinar}

To improve on the existing degree bound on Putinar's Positivstellensatz on a general compact semialgebraic set, we leverage  the following degree bound on Putinar's Positivstellensatz on the hypercube.

\begin{theorem}[{\citep[][Thm. 3]{baldi2024degree}}]
\label{thm:PutBox}
 Let $f \in \R[\bx]$ with $\deg f = d$. If $f > 0$ on $[-1,1]^n$, then there exists an absolute constant $c > 0$ such that:
\[
f \in \cQ(1 \pm \bx)_{rn} \text{ whenever }
r \ge c \log (n) d^2 \kappa_{[-1,1]^n}(f) + \bigO(\kappa_{[-1,1]^n}(f)^{1/2}).
\]
\end{theorem}

With Theorem~\ref{thm:PutBox} at hand, we can now use the lift-and-project procedure (Theorem~\ref{thm:enhLift}) to obtain degree bounds for Putinar's {\germanF} for a polynomial $f >0$ on $S_{\bg}$. The degree bound we obtain in this form (\Cref{thm:ourputi}) improves the dependance of the degree bound on the condition measure $\kappa_{S_{\bg}}(f)$ in the current best bound~\citep[][Cor. 3.3]{baldi2025lojasiewicz} from $\bigO(\kappa_{S_{\bg}}(f)^{7\Ll_\bg +3})$  to $\bigO(\kappa_{S_{\bg}}(f)^{2\Ll_\bg})$.

\begin{theorem}[Effective Putinar {\germanF}]
\label{thm:ourputi}
Let $f \in \R[\bx]$ with $\deg f = d$, and $g_1,\dots,g_m \in \R[\bx]$ such that $1 - \|\bx\|^2_2 \in  \cQ(\bg)$. If $f > 0$ on $S_\bg$, then there is a constant $C_\bg > 0$ depending only on $\bg$ such that
\[
f \in \cQ(\bg)_{rn} \text{
whenever }
r \ge  C_\bg \log (n) d^{4\Ll_\bg+2}\kappa_{S_\bg}(f)^{2\Ll_\bg} + \bigOg( d^{2\Ll_\bg}\kappa_{S_\bg}(f)^{\Ll_\bg})).
\]
\end{theorem}
\begin{proof}
Take $\newell_\bg$, depending only on $\bg$,  such that
$
1 - \|\bx\|^2_2 \in \cQ(\bg)_{\newell_\bg}$.
Using that for any $i \in \{1,\dots,n\}$,
\[
1 \pm \bx_i
= \frac12\Bigl(1- \|\bx\|^2_2 + (1\pm \bx_i)^2 +\sum_{j\neq i} \bx_j^2\Bigr),
\]
we have
\begin{equation}\label{eq:Mbound}
1 \pm \bx  \subset \cQ(\bg)_{\newell_\bg}.
\end{equation}
From \Cref{lem:norm1f-f}\ref{normfit:three}, we have that for $j=1,\dots,m$,
\begin{equation}\label{eq:Nbound}
\normone{g_j} - g_j(\bx) \subset \cQ(1 \pm \bx)_{2d_\bg+1} \subseteq \cQ(\bg)_{\newell_\bg+2d_\bg+1},
\end{equation}
where the second inclusion follows from~\eqref{eq:Mbound} and~\Cref{lem:SubsQTR}\ref{it:lem1}.
For $j=1,\dots,m$, let
\[
\bar g_j(\bx) := \frac{g_j(\bx)}{\normone{g_j}}.
\]
If follows that
\[
\normsup{\bar g_j} = \frac{\normsup{g_j}}{\normone{g_j}} \le 1.
\]
From~\eqref{eq:Nbound} we also have
\begin{equation}\label{eq:Puttrans}
1- \bbg \subseteq \cQ(\bg)_{\newell_\bg+2d_\bg+1}.
\end{equation}
Also note that, $1 - \|\bx\|^2_2 \in \cQ(\bg)$ implies that
\[
S_{\bbg} = S_\bg
\subseteq [-1,1]^n.
\]

Let $T = [-1,1]^n \times [0,1]^m$,
then
$\{(x,\barg(x)): x\in S_{\bbg}\} \subseteq T$.
Since  $f > 0$ on $S_{\bbg}$, by Theorem~\ref{thm:enhLift},
there exists $F \in \R[\bx, \bu]$ with $\deg F = \max\{d, 2d_{\bg}\}$, such that
\begin{equation}
\label{eq:Putprojection}
f(\bx) = F(\bx,\barg(\bx)),
\end{equation}
and
\begin{equation}\label{eq:kPut}
\min_{(x,u)\in T}F(x,u) \ge \tfrac 12 f_{\min} \;\;\text{ and }\;\;
\max_{(x,u)\in T} F(x,u) \le (1+ 2m\cl_{\barg}^2)(4d^2 \kappa_{S_{\barg}}(f))^{2\Ll_{\barg}}f_{\min}.
\end{equation}
Now we consider the  polynomial transformations $(\bx,\bu) \mapsto (\bx,\frac {1+\bu}2)$ mapping $[-1,1]^{n+m}$ to $T$ (of degree $1$) and its inverse  $(\bx,\bu) \mapsto (\bx,2\bu-1)$.
By \Cref{lem:trans}\ref{it:trans1}, $F(\bx,\bu) > 0$ on $T$ implies $F(\bx,\frac {1+\bu}2) > 0$ on $[-1,1]^{n+m}$. Thus, it follows from Theorem~\ref{thm:PutBox} that there is an absolute constant $c >0$, such that taking  \begin{equation}\label{eq:rPut}
r
= 4c d_{\barg}^2 d^2 \log (n+m) \tilde{\kappa}
+ \bigO(\tilde{\kappa}^{1/2}) \;\text{ with }\; \tilde{\kappa} := \frac {\max_{(x,u) \in [-1,1]^{n+m}} F(x,\frac {1+u}2)}{\min_{(x,u) \in [-1,1]^{n+m}} F(x,\frac {1+u}2)},
\end{equation}
we have
\[
F(\bx,\tfrac {1+\bu}2)
\in \cQ(1 \pm \bx, 1 \pm \bu)_{r(n+m)}.
\]
Thus, by \Cref{lem:trans}\ref{it:trans2}
\begin{equation}\label{eq:Fputback}
F(\bx,\bu)
\in \cQ(1 \pm \bx, 1 \pm (2\bu-1))_{r(n+m)} = \cQ(1 \pm \bx, \bu, 1-\bu)_{r(n+m)}.
\end{equation}
Substituting $\bu \leftarrow \barg$ in~\eqref{eq:Fputback} and using~\eqref{eq:Putprojection}, we obtain
\[
  f(\bx) = F(\bx,\barg(\bx)) \in   \cQ(1 \pm \bx, \barg, 1-\barg)_{r(n+m)}.
\]
Therefore,  by~\eqref{eq:Mbound},~\eqref{eq:Puttrans}, and \Cref{lem:SubsQTR}\ref{it:lem1}
\[
  f(\bx) \in \cQ(\bg)_{\newell_\bg+2d_\bg + 1+r(n+m)}.
\]
To conclude the proof notice that in~\eqref{eq:rPut}, 
\begin{align*}
\tilde{\kappa} &= \frac {\max_{(x,u) \in T} F(x,u)}{\min_{(x,u) \in T} F(x,u)}\\
& \le 2(1+ 2m\cl_{\bbg}^2)(4d^2 \kappa_{S_\bg}(p))^{2\Ll_{\bbg}} &\text{(by \eqref{eq:kPut})} \\
& \le  2(1+ 2m(\max_{j=1,\dots,m} \normone{g_j})\cl_{\bg}^2)(4d^2 \kappa_{S_\bg}(p))^{2\Ll_{\bg}} &\text{(by Lemma~\ref{lem:scaled-loj})}
\end{align*}
\end{proof}

\subsection{Effective Schm\"{u}dgen Positivstellensatz}
\label{sec:Schmu}

To improve the degree bound of Schm\"{u}dgen's Positivstellensatz on a general compact semialgebraic set, we leverage  the following degree bound on Schm\"{u}dgen's Positivstellensatz on the hypercube.
\begin{theorem}[{\cite[][Cor. 3]{laurent2023effective}}]
\label{thm:SchBox}
Let $f \in \mathbb{R}[\bx]$ with $\deg f = d$. If $f > 0$ in $[-1,1]^n$, Then
\[
f \in \cT(1 \pm x)_{rn}
\text{ whenever } r \ge C(n,d)\kappa_{[-1,1]^n}(p)^{1/2} + \bigO(dn^{1/2})   ,
\]
where $C(n,d) > 0$ depends only on $n$ and $d$, and may be chosen to be polynomial in $n$ for fixed $d$, or polynomial in $d$ for fixed $n$.
In particular,
\begin{equation}
\label{eq:Schconst}
C(n, d)\le \pi n^{1/2}({\scriptstyle \sqrt{2}}(d + 1))^{\frac n2 + 1 } \text{ and } C(n, d) \le 2\pi ({\scriptstyle \sqrt{2}}(n + 1))^{\frac {d+1}2}.
\end{equation}
\end{theorem}

With Theorem~\ref{thm:SchBox} at hand, we can now use the lift-and-project  procedure (Theorem~\ref{thm:enhLift}) to obtain degree bounds for {\schmu}'s {\germanF} for a polynomial $f >0$ on $S_{\bg}$. The degree bound we obtain in this form (\Cref{thm:ourschmu}) has a dependance on the condition measure $\kappa_{S_{\bg}}(f)$ of $\bigO(\kappa_{S_{\bg}}(f)^{\Ll_\bg})$, where $\Ll_\bg$ is the {\L}ojasiewicz exponent associated with the set $S_\bg$ (see \eqref{eq:LojDist}). Pairing this result with the one obtained for Puitnar's {\germanF}
(\Cref{thm:ourputi}), we see that the degree bound on {\schmu}'s {\germanF} elicits the same quadratic improvement that is elicited between the degree bounds for Putinar's and {\schmu}'s {\germansF} over the hypercube $[-1,1]^n$ (see Table~\ref{tab:bounds}).

\begin{theorem}[Effective Schm\"{u}dgen {\germanF}]
\label{thm:ourschmu} Let $f \in \R[\bx]$ with $\deg f = d$, and $g_1,\dots,g_m \in \R[\bx]$ such that $S_\bg \subset (-1,1)^n$. If~$f > 0$ on $S_\bg$, then there is a constant $C_\bg > 0$ depending only on $\bg$ such that
\[
f \in \cT(\bg)_{rn} \text{ whenever } r \ge  C_{\bg} C(n+m,d+2d_{\bg})(d^2\kappa_{S_\bg}(f))^{\Ll_\bg} + \bigOg(d n^{1/2}),
\]
where $C(\cdot,\cdot) > 0$ is bounded in~\eqref{eq:Schconst}.
\end{theorem}
\begin{proof} By compactness of $S_\bg \subset (-1,1)^n$, it follows from
\Cref{thm:certificates}\ref{thm:schmudgen} that there is $\newell_\bg$, depending only on~$\bg$, such that
\begin{equation}
\label{eq:MSbound}
1 \pm \bx \subset \cT(\bg)_{\newell_\bg}.
\end{equation}
From \Cref{lem:norm1f-f}\ref{normfit:two}, we have that for $j=1,\dots,m$,
\begin{equation}\label{eq:NSbound}
\normone{g_j} - g_j(\bx) \subset \cT(1 \pm \bx)_{d_\bg} \subseteq \cT(\bg)_{\newell_\bg d_\bg},
\end{equation}
where the second inclusion follows from~\eqref{eq:MSbound} and \Cref{lem:SubsQTR}\ref{it:lem2}.
For $j=1,\dots,m$, let
\[
\bar g_j(\bx) := \frac{g_j(\bx)}{\normone{g_j}}.
\]
If follows that
\[
\normsup{\bar g_j} = \frac{\normsup{g_j}}{\normone{g_j}} \le 1.
\]
From~\eqref{eq:NSbound} we also have
\begin{equation}\label{eq:Strans}
1- \bbg \subseteq \cT(\bg)_{\newell_\bg d_\bg}.
\end{equation}
Also note that
\[
S_{\bbg} = S_\bg
\subseteq [-1,1]^n.
\]

Let $T = [-1,1]^n \times [0,1]^n $, then  $\{(x,\barg(x)): x\in S_{\bbg}\} \subseteq T$.
Since  $f>0$ on $S_{\bbg}$, by Theorem~\ref{thm:enhLift},
there is  $F \in \R[\bx, \bu]$ with  $\deg F = \max\{d, 2d_\bg\} \le d+2d_{\bg}$, such that
\begin{equation}
\label{eq:Schprojection}
f(\bx) = F(\bx,\barg(\bx)),
\end{equation}
and
\begin{equation}\label{eq:minmaxFSch}
\min_{(x,u)\in T}F(x,u) \ge \tfrac 12 f_{\min} \text{ and } \max_{(x,u)\in T} F(x,u)  \le (1+ 2m\cl_{\barg}^2)(4d^2 \kappa_{S_{\barg}}(f))^{2\Ll_{\barg}}f_{\min}.
\end{equation}
Now we consider the  polynomial transformations $(\bx,\bu) \mapsto (\bx,\frac {1+\bu}2)$ mapping $[-1,1]^{n+m}$ to $T$ (of degree $1$) and its inverse  $(\bx,\bu) \mapsto (\bx,2\bu-1)$.
By \Cref{lem:trans}\ref{it:trans1}, $F(\bx,\bu) > 0$ on $T$ implies $F(\bx,\frac {1+\bu}2) > 0$ on $[-1,1]^{n+m}$. Thus, it follows from Theorem~\ref{thm:SchBox} that there is a constant $C(n+m,d+2d_{\bg}) >0$ (see~\eqref{eq:Schconst}), such that taking
\begin{equation}\label{eq:rSch}
\begin{gathered}
r =  C(n+m,d+2d_{\bg})\tilde{\kappa}^{1/2} + \bigO((d+2d_{\bg})(n+m)^{1/2}) \\
\text{with }
\tilde{\kappa} := \frac {\max_{(x,u) \in [-1,1]^{n+m}} F(x,\frac {1+u}2)}
{\min_{(x,u) \in [-1,1]^{n+m}} F(x,\frac {1+u}2)},
\end{gathered}
\end{equation}
we have
\begin{equation*}
F(\bx,\tfrac {1+\bu}2) \in \cT(1 \pm \bx, 1 \pm \bu)_{r(n+m)}.
\end{equation*}
Thus, by \Cref{lem:trans}\ref{it:trans2}
\begin{equation}\label{eq:FSchback}
F(\bx,\bu)
\in \cT(1 \pm \bx, 1 \pm (2\bu-1))_{r(n+m)} = \cT(1 \pm \bx, \bu, 1-\bu)_{r(n+m)}.
\end{equation}
Substituting $\bu \leftarrow \barg$ in~\eqref{eq:FSchback} and using~\eqref{eq:Schprojection}, we obtain
\[
  f(\bx) = F(\bx,\barg(\bx)) \in  \cT(1 \pm \bx, \barg, 1-\barg)_{r(n+m)}.
\]
Therefore,  by~\eqref{eq:MSbound},~\eqref{eq:Strans}, and \Cref{lem:SubsQTR}\ref{it:lem2},
\[
  f(\bx) \in \cT(\bg)_{\newell_\bg d_\bg r(n+m) }.
\]
To conclude the proof, notice in~\eqref{eq:rSch}
\begin{align*}
\tilde{\kappa} &= \frac {\max_{(x,u) \in T} F(x,u)}{\min_{(x,u) \in T} F(x,u)}\\
& \le 2(1+ 2m\cl_{\bbg}^2)(4d^2 \kappa_{S_\bg}(f))^{2\Ll_{\bbg}} &\text{(by~\eqref{eq:minmaxFSch})} \\
& = 2(1+ 2m(\max_{j=1,\dots,m} \normone{g_j})\cl_{\bg}^2)(4d^2 \kappa_{S_\bg}(p))^{2\Ll_{\bg}}. &\text{(by \Cref{lem:scaled-loj})}
\end{align*}
\end{proof}

\subsection{Effective Krivine-Stengle and the extended-Handelman Po\-si\-tiv\-ste\-llen\-s\"atze}
\label{sec:both}

Previously, we focused on degree bounds for {\germansF} based on SOS polynomials. We now analyze degree bounds for {\germansF} based on non-negative constants, a topic largely unexplored in the literature. This gap likely exists because hierarchies employing such {\germansF} exhibit poor convergence~\citep[see, e.g.,][]{lasserre2013lagrangian, PenaVZ07}. However, the recent work in~\cite{kuryatnikova2024reducing} motivates formally studying degree bounds for this class of certificates.

Existing degree bounds for Handeman's Theorem on the hypercube are written in terms of the norm $L(f)$. In \Cref{thm:HandBox} we state this result in terms of the $\normsup{f}$. To do this we first bound $L(f)$ using $\normsup{f}$. 
\begin{lemma}\label{lem:Lftosup}
 Let \( f \in \R[\bx]\) with  \(\deg f = d\).  Then
\[
L(f)
\le
60^d \|f\|_{\sup,[0,1]^n}.
\]
\end{lemma}
\begin{proof}
We first consider the homogeneous case. Let
\(
p(\bx)=\sum_{\alpha \in \N^n: \|\alpha\|_1=k} p_\alpha \bx^\alpha
\)
be a homogeneous polynomial  with \(\deg p = k\). Let \(A\in \operatorname{Sym}^k(\mathbb{R}^n)\) be the associated symmetric \(k\)-linear form, so that
\(
p(\bx)=A(\bx,\dots,\bx).
\)
For every multi-index \(\alpha \in \N^n\) with \(\|\alpha\|_1=k\), we have
\begin{equation}\label{eq:Avsp}
p_\alpha \frac{\alpha!}{k!} =
A(
\underbrace{e_1,\ldots,e_1}_{\alpha_1},
\ldots,
\underbrace{e_n,\ldots,e_n}_{\alpha_n}
)
\le
\sup_{\bx_1,\ldots,\bx_k \in [0,1]^n}
|A(\bx_1,\ldots,\bx_k)|.
\end{equation}
We use the finite-difference polarization identity
\[
k!A(x_1,\ldots,x_k)
=
\sum_{S\subseteq\{1,\ldots,k\}}
(-1)^{k-|S|}
p\left(\sum_{j\in S}x_j\right).
\]
We obtain
\begin{align*}
k!|A(x_1,\ldots,x_k)|
&\le
\sum_{S\subseteq\{1,\ldots,k\}}
\left|
p\left(\sum_{j\in S}x_j\right)
\right|\\
&= \sum_{S\subseteq\{1,\ldots,k\}}
k^k
\left|
p\left(\frac{1}{k}\sum_{j\in S}x_j\right)
\right| \\
&\le
(2k)^k\|p\|_{\sup,[0,1]^n},
\end{align*}
where we have used  that \(x_j\in[0,1]^n\) implies $\frac{1}{k}\sum_{j\in S}x_j \in [0,1]^n$ for all $S\subseteq\{1,\ldots,k\}$.

Using \(k^k/k!\le e^k\), we get
\[
|A(x_1,\ldots,x_k)|
\le
(2e)^k\|p\|_{\sup,[0,1]^n}.
\]
And using~\eqref{eq:Avsp} yields
\[
L(p)
\le
(2e)^k\|p\|_{\sup,[0,1]^n}.
\]
Now, consider the the homogeneous decomposition of \(f\)
\[
f(x)=\sum_{|\alpha|\le d}f_\alpha x^\alpha
=
\sum_{k=0}^d p_k(x).
\]
We claim that for every \(0\le k\le d\),
\begin{equation}\label{eq:Lfclaim}
\|p_k\|_{\sup,[0,1]^n}
\le
(4e)^d
\|f\|_{\sup,[0,1]^n}.
\end{equation}
Using the claim we obtain 
\[L(f) \le \max_{k=0,\dots,d} L(p_k) \le \max_{k=0,\dots,d} (2e)^k\|p_k\|_{\sup,[0,1]^n} \le (8e^2)^d
\|f\|_{\sup,[0,1]^n}.\]

Now to prove Claim~\eqref{eq:Lfclaim}, fix \(x\in[0,1]^n\), and define the univariate polynomial
\[
q_x(t):=f(tx) =\sum_{k=0}^d p_k(x)t^k.
\]
By using the  the Lagrange interpolation formula applied at the equidistant nodes $0,1/d,\dots,1$ together with elementary coefficient bounds for the Lagrange basis polynomials, we obtain the following univariate coefficient estimate: if
\[
q(t)=\sum_{k=0}^d a_kt^k
\text{ then }
|a_k|\le (4e)^d  \|q\|_{\sup,[0,1]^n}
\qquad \text{for all }0\le k\le d.
\]
Applying this estimate to \(q_x\), we get
\[
|p_k(x)|
\le (4e)^d  \|q_x\|_{\sup,[0,1]^n}
\le (4e)^d  \|f\|_{\sup,[0,1]^n}.
\]
Taking the supremum over \(x\in[0,1]^n\) yields Claim~\eqref{eq:Lfclaim}. 
\end{proof}

We leverage  the following degree bound on Handelman's Positivstellensatz on the hypercube, that readily follows from~\citep[][Thm. 1.4]{DeKlerkLaurent2010},  to obtain degree bounds on the
extended-Handelman's {\germanF}.

\begin{theorem}\label{thm:HandBox}
Let $f \in \mathbb{R}[\bx]$ with $\deg f = d$. If $f > 0$ in $[-1,1]^n$, Then
\[
f \in \cR(1\pm\bx)_{rn}
\text{
whenever }
r \ge 60^d  d^3 n^{d} \kappa_{[-1,1]^n}(f).
\]
\end{theorem}
\begin{proof}
Consider the  polynomial transformations $\bx \mapsto 2\bx-1$ mapping $[0,1]^{n}$ to $[-1,1]^n$  and its inverse  $\bx \mapsto \tfrac{1+\bx}{2}$.
By \Cref{lem:trans}\ref{it:trans1}, $f(\bx) > 0$ on $[-1,1]^n$ implies $\tilde f(\bx) := f(2\bx-1) > 0$ on $[0,1]^n$.
Thus, by~\citep[][Thm. 1.4]{DeKlerkLaurent2010}
we have
\[
\tilde f(\bx) \in \cR(\bx, 1-\bx)_{rn} \text{ whenever } r \ge \binom{d+1}{3} n^d \frac{L(\tilde f)}{\min\{\tilde f(x): x \in [0,1]^{n}\}},
\]
where $L(\tilde f)$ is defined in~\eqref{eq:wnorm}.
Therefore, by \Cref{lem:trans}\ref{it:trans2},
\begin{equation}\label{eq:FHandback}
f(\bx) \in \cR(\tfrac{1+\bx}{2}, 1 - \tfrac{1+\bx}{2})_{rn} = \cR(1 \pm \bx)_{rn}.
\end{equation}
The result then follows from
\[
\min\{\tilde f(x): x \in [0,1]^{n}\} = \min\{f(x): x \in [-1,1]^{n}\} \text{ and }
\binom{d+1}{3} \le d^3,
\]
combined with
\[ L(\tilde f) \le 60^d \max\{\tilde f(x): x \in [0,1]^n\} =  60^d  \|f\|_{\sup}\]
that follows by \Cref{lem:Lftosup}.
\end{proof}

With Theorem~\ref{thm:HandBox} at hand, we can now use the lifting procedure (Theorem~\ref{thm:enhLift}) to obtain an effective version of the extended-Handelman {\germanF}  (\Cref{thm:certificates}\ref{it:nosos}), whose proof is analogous to the one of \Cref{thm:ourputi}.
The degree bound we obtain in this form (\Cref{thm:genHand}) has a dependance on the condition measure $\kappa_{S_{\bg}}(f)$ of $\bigO(\kappa_{S_{\bg}}(f)^{2\Ll_\bg})$, where~$\Ll_\bg$ is the {\L}ojasiewicz exponent associated with the set $S_\bg$ (see \eqref{eq:LojDist}). Pairing this result with the one obtained for Puitnar's {\germanF}
(\Cref{thm:ourputi}), we see that we obtain a degree bound for the extended-Handelman {\germanF} that elicits the same dependance on $\kappa_{S_{\bg}}(f)$ that is elicited by the degree bound obtained for the Putinar {\germanF}. This is notable because Putinar's {\germanF} uses SOS to certify nonnegativity, leading to a hierarchy of SDP approximations, while the extended-Handelman {\germanF} uses nonegativy constants (SOS of degree zero) to certify nonnegativity, leading to a hierarchy of LP approximations.

\begin{theorem}[Effective extended-Handelman {\germanF}]\label{thm:genHand}
Let $f \in \R[\bx]$ with $\deg f = d$, and $g_1,\dots,g_m \in \R[\bx]$ such that $S_\bg \subset [-1,1]^n$. If~$f > 0$ on $S_\bg$, then there is a constant $C_\bg > 0$ depending only on $\bg$ such that
\[
f \in \cR(1\pm \bx, \bg)_{r(n+m)},
\text{ whenever }
r =  C_\bg 60^d d^{4\Ll_\bg + 3} n^{2d_\bg d} \kappa_{S_\bg}(f)^{2\Ll_\bg}.
\]
\end{theorem}

\begin{proof}
From \Cref{lem:norm1f-f}\ref{normfit:one}, we have that for $j=1,\dots,m$,
\begin{equation}\label{eq:NHbound}
\normone{g_j} - g_j(\bx) \subset \cR(1 \pm \bx)_{d_\bg}.
\end{equation}
Let
\[
\bar g_j(\bx) := \frac{g_j(\bx)}{\normone{g_j}},
\]
for $j=1,\dots,m$. If follows that
\[
\normsup{\bar g_j} = \frac{\normsup{g_j}}{\normone{g_j}} \le 1,
\]
for $j=1,\dots,m$.
Further, from~\eqref{eq:NHbound} we  have
\begin{equation}\label{eq:Htrans}
1- \bbg \subseteq \cR(\bg)_{d_\bg}.
\end{equation}
Also note that
\[
S_{\bbg} = S_\bg
\subseteq [-1,1]^n.
\]
Let $T = [-1,1]^n \times [0,1]^n $, then  $\{(x,\barg(x)): x\in S_{\bg}\} \subseteq T$.
Since  $f>0$ on $S_{\bg}$, by Theorem~\ref{thm:enhLift},
there is  $F \in \R[\bx, \bu]$ with  $\deg F = \max\{d, 2d_\bg\} \le d+2d_{\bg}$, such that
\begin{equation}
\label{eq:Handprojection}
f(\bx) = F(\bx,\barg(\bx)),
\end{equation}
and
\begin{equation}\label{eq:minmaxFHand}
\min_{(x,u)\in T}F(x,u) \ge \tfrac 12 f_{\min} \text{ and } \max_{(x,u)\in T} F(x,u)  \le (1+ 2m\cl_{\barg}^2)(4d^2 \kappa_{S_{\barg}}(f))^{2\Ll_{\barg}}f_{\min}.
\end{equation}
Now we consider the  polynomial transformations $(\bx,\bu) \mapsto (\bx,\frac {1+\bu}2)$ mapping $[-1,1]^{n+m}$ to $T$ (of degree $1$) and its inverse  $(\bx,\bu) \mapsto (\bx,2\bu-1)$.
By \Cref{lem:trans}\ref{it:trans1}, $F(\bx,\bu) > 0$ on $T$ implies $F(\bx,\frac {1+\bu}2) > 0$ on $[-1,1]^{n+m}$. Thus, it follows from \Cref{thm:HandBox} that taking
\begin{equation}\label{eq:rHand}
\begin{gathered}
r = 60^{d+2d_\bg} (d + 2d_{\bg})^{3} (n+m)^{d + 2d_{\bg}}\tilde{\kappa}\\
\text{ with }
\tilde{\kappa} := \frac {\max_{(x,u) \in [-1,1]^{n+m}} F(x,\frac {1+u}2)}
{\min_{(x,u) \in [-1,1]^{n+m}} F(x,\frac {1+u}2)},
\end{gathered}
\end{equation}
we have
\begin{equation*}
F(\bx,\tfrac {1+\bu}2) \in \cR(1 \pm \bx, 1 \pm \bu)_{r(n+m)}.
\end{equation*}
Thus, by \Cref{lem:trans}\ref{it:trans2}
\begin{equation}\label{eq:Handback}
F(\bx,\bu)
\in \cR(1 \pm \bx, 1 \pm (2\bu-1))_{r(n+m)} = \cR(1 \pm \bx, \bu, 1-\bu)_{r(n+m)}.
\end{equation}

Substituting $\bu \leftarrow \barg$ in~\eqref{eq:Handback} and using~\eqref{eq:Handprojection}, we obtain
\[
  f(\bx) = F(\bx,\barg(\bx)) \in  \cR(1 \pm \bx, \barg, 1-\barg)_{r(n+m)}.
  \]
Therefore,  by~\eqref{eq:Htrans} and \Cref{lem:SubsQTR}\ref{it:lem3},
\[
  f(\bx) \in \cR(1 \pm \bx, \bg)_{d_\bg r(n+m)}.
\]
To conclude the proof, notice that in~\eqref{eq:rHand}
\begin{align*}
\tilde k &= \frac {\max_{(x,u) \in T} F(x,u)}{\min_{(x,u) \in T} F(x,u)} \\
&\le 2(1+ 2m\cl_{\bar \bg}^2)(4d^2 \kappa_{S_\bg}(p))^{2\Ll_{\bar \bg}} &\text{(by~\eqref{eq:minmaxFHand})}\\
& = 2(1+ 2m(\max_{j=1,\dots,m} \normone{g_j})^2\cl_{\bg}^2)(4d^2 \kappa_{S_\bg}(p))^{2\Ll_{\bg}} &\text{(by \Cref{lem:scaled-loj})}.
\end{align*}
\end{proof}

As shown next, the degree bound for the extended-Handelman's {\germanF} readily provides a degree bound for the Krivine-Stengle {\germanF}.

\begin{corollary}[Effective Krivine-Stengle {\germanF}]\label{thm:degbdKS}
Let $f \in \R[\bx]$ with $\deg f = d$, and $g_1,\dots,g_m \in \R[\bx]$ such that $S_\bg \subset [-1,1]^n$.
Assume $1-g_j \ge 0$ on $S_\bg$ for $j=1,\dots,m$ and that
$\{1,g_1,\dots,g_m\}$ generates $\mathbb{R}[\bx]$.
If~$f > 0$ on $S_\bg$, then there is a constant $C_\bg > 0$ depending only on $\bg$ such that
\[
f \in \cR(\bg, 1-\bg)_{r(n+m)},
\text{ whenever }
r =  C_\bg 2^d  d^{4\Ll_\bg + 3}n^{2d_\bg d}\kappa_{S_\bg}(f)^{2\Ll_\bg}.
\]
\end{corollary}
\begin{proof}
By \Cref{thm:genHand} we have, $f  \in \cR(1 \pm \bx, \bg)_{r(n+m)}$ for $ r \ge C'_\bg 2^d n^{2d_\bg d} d^{4\Ll_\bg + 3}\kappa_{S_\bg}(f)^{2\Ll_\bg}$ for some $C'_\bg > 0$ depending only on $\bg$.
By the Theorem~\ref{thm:certificates}\ref{it:KS}, there is $\newell_\bg$ depending only on $\bg$ such that $1\pm \bx \subset \cR(\bg, 1- \bg)_{\newell_\bg}$ and therefore, by \Cref{lem:SubsQTR}, $f \in \cR(\bg, 1- \bg)_{\newell_\bg+r(n+m)}$.
\end{proof}

\section{Final remarks}
\label{sec:final}
The identity~\eqref{eq:Hopt}; that is, \(
H(x) = \min_{u \in [0,1]^m} |u-\bg(x)|2^2\) for all \(x \in [-1,1]^n\),
used in the proof of \Cref{thm:enhLift} shows that $|u-\bg(x)|2^2 =\sum{j=1}^m(u_j-\bar g_j(x))^2$ is a polynomial lifting of the semialgebraic function $H(\bx)$. A similar identity involving $G(\bx)$ (see~\eqref{eq:Gdef}) will allow us to drop the $2$ in the exponent obtained in \Cref{thm:enhLift}\ref{it:T3}.
Such a result will be optimal, in the sense that this exponent can not be smaller than $\Ll_\bg$. Otherwise, by taking $\bg = 1\pm \bx$, and using that $\Ll_\bg=1$ in this case, one can obtain improvements ``for free" on the degree bounds for the {\germansF} on the hypercube, unless the degree bound does not depend on the condition number of $f$. Having such an optimal exponent will imply that for the instance class where $\Ll_\bg =1$, the error bounds have the same exponent on the condition number, as in the hypercube. As mentioned earlier, $\Ll_\bg =1$ holds under common regularity assumptions~\citep[][Thm.~2.11]{baldi2023effective}.

We focused on the condition number $\kappa_{S_\bg}(f) = \normsup{f}/f_{\min}$ of a polynomial $f > 0$ on $S_{\bg}$. Recent work gives degree bounds in terms of condition numbers of the form $\|f\|_{\coeff}/f_{\min}$, where $\|f\|_{\coeff}$ denotes a norm (e.g., $\ell_1$ or $\ell_\infty$) on the coefficient vector in a chosen basis (e.g., monomial or Tchebychev). Because the proof of \cref{thm:enhLift} is constructive, similar bounds for these norms can be obtained as for the sup-norm. Thus, our method can also be applied to those cases.

We leveraged degree bounds for the hypercube to obtain degree bounds for general (compact semialgebraic) sets. However, it is not difficult to see that the proposed lift-and-project methodology will work for other choices of simple sets, such as the ball or the simplex. It will be interesting to see if using a different set will lead to improvements in the degree bounds for general sets. We believe it will be unlikely for this approach to lead to advantages in terms of the dependence of the degree bounds on the condition number of the polynomial of interest, but it may lead to improvement in terms of other constants related to the polynomials involved in the nonnegative certificate.

Finally, as mentioned earlier, the lift-and-project methodology presented here to obtain the degree bounds is motivated by the work in~\cite{kuryatnikova2024reducing}. In there, the results initially obtained for compact semialgebraic sets are extended to unbounded semialgebraic sets. Thus, we expect that by following an approach similar to the one presented here, the lift-and-project methodology can be extended to obtain degree bounds (equivalently, convergence rates) associated with hierarchies to obtain the minimum of a polynomial over an unbounded semialgebraic set, as is the case, for example, in hierarchies designed to solve {\em copositive programs}~\citep{Bomz09}.

\end{document}